\documentclass[onefignum,onetabnum,x11names]{siamart171218}

\newsiamremark{remark}{Remark}
\newsiamremark{hypothesis}{Hypothesis}
\crefname{hypothesis}{Hypothesis}{Hypotheses}
\newsiamthm{claim}{Claim}

\usepackage{lipsum}
\usepackage{amsfonts}
\usepackage{graphicx}
\usepackage{epstopdf}
\usepackage{algorithmic}
\ifpdf
  \DeclareGraphicsExtensions{.eps,.pdf,.png,.jpg}
\else
  \DeclareGraphicsExtensions{.eps}
\fi
\usepackage{hyperref}
\usepackage[utf8]{inputenc}
\usepackage[english]{babel}
\usepackage{stmaryrd}
\usepackage{enumitem}
\usepackage{eufrak}

\newcommand{\hlift}[1]{#1^{\bar{h}}}
\newcommand{\vlift}[1]{#1^{\bar{v}}}
\newcommand{\hcomp}[1]{d\pi\left(#1\right)}
\newcommand{\vcomp}[1]{K\left(#1\right)}
\newcommand{\metric}[3][]{g_{#1}\left(#2, #3\right)}
\newcommand{\bmetric}[3][]{\bar{g}_{#1}\left(#2,#3\right)}
\newcommand{\liebracket}[2]{\left[#1,#2\right]}
\newcommand{\bliebracket}[2]{\left\llbracket #1,#2 \right\rrbracket}

\headers{Non-Natural Metrics on the Tangent Bundle}{B. Vang, R. Tron}

\title{Non-Natural Metrics on the Tangent Bundle\thanks{}
  \funding{This work was funded by the National Science Foundation grant NSF CMMI-1728277.}}

\author{Bee Vang\thanks{Boston University, Boston, MA 
  (\email{bvang@bu.edu}
).}
\and Roberto Tron\thanks{Boston University, Boston, MA
  (\email{tron@bu.edu}).}}

\ifpdf
\hypersetup{
  pdftitle={Non-Natural Metrics on the Tangent Bundle},
  pdfauthor={B. Vang, R. Tron}
}
\fi

\begin{document}

\maketitle

\begin{abstract}
Natural metrics provide a way to induce a metric on the tangent bundle from the metric on 
its base manifold. The most studied type is the Sasaki metric, which applies the base 
metric separately to the vertical and horizontal components. We study a more general class
of metrics which introduces interactions between the vertical and horizontal components,
with scalar weights. Additionally, we explicitly clarify how to apply our and other
induced metrics on the tangent bundle to vector fields where the vertical component is
not constant along the fibers. We give application to the Special Orthogonal Group
SO(3) as an example.
\end{abstract}



\section{Introduction}
The study of tangent bundles and their relationship to the base manifold often
rely on the Sasaki metric. However, we may gain valuable mathematical 
and physical insights by choosing a more general metric. For mechanical systems,
tangent bundles arise naturally where the manifold is the configuration space
and the Lagrangian mechanics involve the configurations and their velocities
as state variables \cite{BulloFrancesco2005Gcom}. A fundamental process is damping, 
in which the changes on the configuration depend on changes to its velocity. 
Hence, we want to study metrics where this kind of interaction is considered.
The Sasaki metric does not consider these kind of interactions.

The contributions of this paper are the generalization of the Sasaki metric 
and the derivation of the corresponding Levi-Civita connection on the
tangent bundle. We also clarify the application of the results
to general vector fields that are not constant along fibers.

The paper is organized as follows. A brief overview of relevant differential
geometry concepts are provided in \cref{sec:background}, our main results are 
in \cref{sec:main1} and \cref{sec:main2}, examples are in \cref{sec:experiments}, 
and the conclusions follow in \cref{sec:conclusions}.

\section{Background}
\label{sec:background}
Let $M$ be a n-dimensional differentiable manifold equipped with a Riemannian metric $g$ and $TM$ the 
tangent bundle of $M$. For a point $p \in M$, let $T_{p}M$ denote the 
tangent space of $M$ at $p$. A point $\bar{P} \in TM$ is a pair in the set
$\left\{ \left(p,u\right) \mid p \in M, u \in T_{p}M\right\}$. 
Let $\pi: TM \rightarrow M$ be the projection map. 
The differential of the projection map is a smooth map denoted as $d\pi: TTM \rightarrow TM$. 
For any vector fields $X,Y \in \mathfrak{X}\left(M\right)$,
the Levi-Civita connection on $M$ is denoted by $\nabla_{X}Y$.

From Sasaki, the tangent space $T_{\bar{P}}TM$ is a direct sum decomposition 
$T_{\bar{P}}TM = \mathcal{H}_{\bar{P}} \bigoplus \mathcal{V}_{\bar{P}}$,
where $\mathcal{H}_{\bar{P}}$ is the horizontal subspace and $\mathcal{V}_{\bar{P}}$ is the
vertical subspace \cite{Sasaki1958}. To construct the subspaces, we begin by defining
the exponential map on $M$. For an open neighborhood $U$ of $p := \pi\left(\bar{P}\right)\in M$,
the exponential map $exp_{p}: T_{p}M \rightarrow M$ maps a neighborhood $U'$ of 0
in $T_{p}M$ diffeomorphicly onto $U$. Let $\tau: \pi^{-1}\left(U\right) \rightarrow T_{p}M$ be
the smooth map which parallel transports every $Y \in \pi^{-1}\left(U\right)$ from
$q = \pi\left(Y\right)$ to $p$.
For $u \in T_{p}M$, let $R_{-u}: T_{p}M \rightarrow T_{p}M$ be the translation defined 
by $R_{-u}\left(X\right) = X - u$ for $X \in T_{p}M$. Then, the connection map
$K_{\left(p,u\right)}:T_{\left(p,u\right)}TM \rightarrow T_{p}M$ corresponding to the Levi-Civita connection
is defined as
\begin{equation}
  \label{eq:connectionMap}
  K\left(\bar{X}\right)_{\left(p,u\right)}= d\left(exp_{p}\circ R_{-u}\circ \tau \right)\left(\bar{X}\right)_{p} 
\end{equation}
for all $\bar{X} \in T_{\left(p,u\right)}TM$. The vertical subspaces is then defined as
the kernel of the differential $d\pi$, while the horizontal subspace is defined as
the kernel of the connection map $K$. Throughout this paper, We will use the
$d\pi$ and $K$ mappings as projections on the horizontal and vertical subspaces.

A curve $\bar{\gamma}: I \rightarrow TM$ in the tangent bundle is said to be horizontal
if its tangent $\bar{\gamma}'(t)$ satisfies $\bar{\gamma}'(t) \in \mathcal{H}_{\bar{\gamma}(t)}$
for all $t \in I$. And similarly, a curve $\bar{\gamma}: I \rightarrow TM$ in the tangent bundle
is said to be vertical if its tangent $\bar{\gamma}'(t)$ satisfies
$\bar{\gamma}'(t) \in \mathcal{V}_{\bar{\gamma}(t)}$ for all $t \in I$.

If $X$ is a vector field on $M$, then there is a unique vector field
$\hlift{X}$ on $TM$ called the horizontal lift of $X$
and a unique vector field $\vlift{X}$ on $TM$ called
the vertical lift of $X$ such that
\begin{equation}
  \label{eq:DefLifts}
  \begin{split}
    &d\pi\left(\hlift{X}\right)_{\bar{P}} = X_{\pi\left(\bar{P}\right)}, 
    \quad K\left(\hlift{X}\right)_{\bar{P}} = 0_{\pi\left(\bar{P}\right)}\\
    &d\pi\left(\vlift{X}\right)_{\bar{P}} = 0_{\pi\left(\bar{P}\right)}, 
    \quad K\left(\vlift{X}\right)_{\bar{P}} = X_{\pi\left(\bar{P}\right)}
  \end{split}
\end{equation}
for all $\bar{P} \in TM$. A result of the tangent space decomposition is that any tangent
vector $\bar{X} \in T_{\bar{P}}TM$ can be decomposed into its horizontal and vertical
components $\bar{X} = \hlift{A} + \vlift{B}$ where
$A = d\pi\left(\bar{X}\right), B = K\left(\bar{X}\right) \in T_{\pi\left(\bar{P}\right)}M$.

It is important to note that the standard results and our results in \cref{sec:main1}
rely on vector fields that only change along horizontal curves. We will denote these
type of vector fields as \textbf{\textit{lift-decomposable vector fields}}.
\begin{definition}
  A vector field $\bar{X} \in \mathfrak{X}\left(TM\right)$ is
  \textbf{\textit{lift decomposable}} if it only changes along horizontal curves.
  Then any vector field $\bar{X}$ can be decomposed locally around $(p,u) \in TM$
  as $\bar{X}_{(p,u)} = \hlift{A}_{(p,u)} +  \vlift{B}_{(p,u)}$ for $A,B \in T_{p}M$.
\end{definition}
\begin{remark}
  Lift-decomposable vector fields are constant along the fibers in the that sense
  $d\pi\left(\bar{X}(p,u)\right)_{(p,u)} = d\pi\left(\bar{X}(p,u')\right)_{(p,u')}$
  and similarly for the connection map $K$ of $\bar{X}$ for any $p \in M, u, u' \in T_{p}M$,
  and $\bar{X} \in \mathfrak{X}\left(TM\right)$.
\end{remark}
In general, lift-decomposable vector fields may be too limiting. In \cref{sec:main2},
we show how to extend the results for lift-decomposable vector fields
to any general vector fields that may change along both horizontal
and vertical curves.

As shown in \cite{Dombrowski1962}, the Lie bracket of horizontal and vertical lifts on $TM$
are given by the following
\begin{equation}
  \label{eq:LieBracketRelations}
  \begin{split}
    \bliebracket{\vlift{X}}{\vlift{Y}}_{\left(p,u\right)} &= 0\\
    \bliebracket{\hlift{X}}{\vlift{Y}}_{\left(p,u\right)} &= \vlift{(\nabla_{X}Y)}_{p}\\
    \bliebracket{\hlift{X}}{\hlift{Y}}_{\left(p,u\right)} &= \hlift{\liebracket{X}{Y}}_{p} 
    - \vlift{(\mathfrak{R}(X,Y)u)}_{p}\\
  \end{split}
\end{equation}
for all vector fields $X,Y \in \mathfrak{X}\left(M\right)$, 
$\left(p,u\right) \in TM$, $\hlift{X}, \vlift{X}, \hlift{Y}, \vlift{Y}$ are the
respective horizontal and vertical lifts, and $\mathfrak{R}$ is the curvature tensor
on $M$. Note that the vector fields are lift decomposable.

A metric $\bar{g}$ on the tangent bundle is said to be \textit{natural} with respect to
$g$ on M if
\begin{equation}
  \begin{split}
    &\bmetric[(p,u)]{\hlift{X}}{\hlift{Y}} = \metric[p]{X}{Y}\\
    &\bmetric[(p,u)]{\hlift{X}}{\vlift{Y}} = 0
  \end{split}
\end{equation}
for all $X,Y \in \mathfrak{X}\left(M\right)$ and $(p,u) \in TM$.
The Sasaki metric, first introduced in \cite{Sasaki1958}, is a special natural
metric that has been widely used to study the relationship between the
base manifold and its tangent bundle. The Sasaki metric is given as
\begin{equation}
  \begin{split}
    &\bmetric[(p,u)]{\hlift{X}}{\hlift{Y}} = \metric[p]{X}{Y}\\
    &\bmetric[(p,u)]{\hlift{X}}{\vlift{Y}} = 0\\
    &\bmetric[(p,u)]{\vlift{X}}{\vlift{Y}} = \metric[p]{X}{Y}.
  \end{split}
\end{equation}

The Kozul formula on $(M,g)$ is given by
\begin{multline}
  \label{eq:KozulFormula}
  2\metric{\nabla_{X}Y}{Z} = X\metric{Y}{Z}
  + Y\metric{X}{Z} - Z\metric{X}{Y}\\
  + \metric{\liebracket{X}{Y}}{Z} 
  - \metric{\liebracket{X}{Z}}{Y} 
  - \metric{\liebracket{Y}{Z}}{X}
\end{multline}
for all vector fields $X, Y, Z \in \mathfrak{X}\left(M\right)$.

\section{Levi-Civita Connection for Lift-Decomposable Vector Fields }
\label{sec:main1}
In this section, we assume that all vector fields on $TM$ are lift decomposable.
We define a non-natural metric $\bar{g}_{\left(p,u\right)}$ on the tangent bundle as
\begin{equation}
  \label{eq:NonNaturalMetric}
  \begin{split}
    &\bmetric[(p,u)]{\hlift{X}}{\hlift{Y}} = m_{1}\metric[p]{X}{Y}\\
    &\bmetric[(p,u)]{\hlift{X}}{\vlift{Y}} = m_{2}\metric[p]{X}{Y}\\
    &\bmetric[(p,u)]{\vlift{X}}{\vlift{Y}} = m_{3}\metric[p]{X}{Y}
  \end{split}
\end{equation}
where $\metric[p]{X}{Y}$ is the metric on the manifold $M$ at
point $p$. $\hlift{X}, \vlift{X}, \hlift{Y}, \vlift{Y}$ are the
respective horizontal and vertical lifts of the vector fields
$X, Y \in \mathfrak{X}\left(M\right)$, $(p,u) \in TM$ and $m_{1}, m_{2}, m_{3} \in \mathbb{R}$.
The scalars $m_{1}, m_{2}, m_{3}$ must be chosen such that $m_{1}, m_{3} > 0$ and
$m_{1}m_{3}-m_{2}^{2} > 0$.
\begin{proposition}
  Given a Riemannian manifold $(M,g)$ and $m_{1},m_{2},m_{3}$ chosen such that
  $m_{1}, m_{3} > 0$ and $m_{1}m_{3}-m_{2}^{2}>0$, the metric $\bar{g}$
  defined in \cref{eq:NonNaturalMetric} is a Riemannian metric on TM.
\end{proposition}
\begin{proof}
  The metric $\bar{g}$ must be an inner product on $T_{(p,u)}TM$ at each point $(p,u) \in TM$.
  The symmetry and linearity properties can be verified through simple calculations. To show
  positive definiteness, we consider a tangent vector $\bar{Z} \in T_{(p,u)}TM$ where
  $\bar{Z} = \hlift{X} + \vlift{Y}$ for $X,Y \in T_{p}M$. Then
  \begin{equation*}
    \bmetric[(p,u)]{\bar{Z}}{\bar{Z}} = m_{1}\metric[p]{X}{X} + 2m_{2}\metric[p]{X}{Y}
    + m_{3}\metric[p]{Y}{Y}.
  \end{equation*}
  The metric can be bound from below by using the Cauchy-Schwarz inequality
  such that
  \begin{equation*}
    \bmetric[(p,u)]{\bar{Z}}{\bar{Z}} \geq m_{1}\|X\|^{2} - 2m_{2}\|X\|\|Y\|
    + m_{3}\|Y\|^{2}
  \end{equation*}
  where $\|\cdot\|$ is the norm with respect to $g$. The above equation
  can be rewritten in matrix notation as
  \begin{equation*}
    \bmetric[(p,u)]{\bar{Z}}{\bar{Z}} \geq
    \begin{bmatrix} \|X\| \\ \|Y\| \end{bmatrix}^{T}
    \begin{bmatrix}
      m_{1} & -m_{2}\\
      -m_{2} & m_{3}
    \end{bmatrix}
    \begin{bmatrix} \|X\| \\ \|Y\| \end{bmatrix}.
  \end{equation*}
  From the above inequality, $\bar{g}$ must be positive definite since
  the middle matrix is positive definite by $m_{1}, m_{3} > 0$ and $m_{1}m_{3}-m_{2}^2 > 0$.
\end{proof}
The metric in \cref{eq:NonNaturalMetric} allows us to choose from a class
of metrics on the tangent bundle with different horizontal and vertical subspaces
along with their Levi-Civita connections.

Using the metric defined in \cref{eq:NonNaturalMetric}, the Kozul
formula in \cref{eq:KozulFormula}, and the relations from \cref{eq:LieBracketRelations},
we derive properties of the corresponding Levi-Civita connection
$\bar{\nabla}$ on $TM$ for horizontal and vertical lifts (the proof closely mirrors those
found in \cite{Dombrowski1962,Gudmundsson2002,Kowalski1971,Sasaki1958}).
\begin{proposition}
  \label{lem:KozulFormula}
  Given a Riemannian manifold (M, g) and its tangent bundle TM equipped
  with the metric in \cref{eq:NonNaturalMetric},
  the Levi-Civita connection $\bar{\nabla}$ on TM satisfies
  \begin{enumerate}[label=(\roman*)]
  \item
    $2\bmetric{\bar{\nabla}_{\hlift{X}}\hlift{Y}}{\hlift{Z}} =
    2m_{1}\metric{\nabla_{X}Y}{Z} 
    + 2m_{2}\metric{\mathfrak{R}(u,X)Y}{Z}$\label{it:LeviProp1}
  \item
    $2\bmetric{\bar{\nabla}_{\hlift{X}}\hlift{Y}}{\vlift{Z}} =
    2m_{2}\metric{\nabla_{X}Y}{Z} 
    - m_{3}\metric{\mathfrak{R}(X,Y)u}{Z}$\label{it:LeviProp2}
  \item
    $2\bmetric{\bar{\nabla}_{\hlift{X}}\vlift{Y}}{\hlift{Z}} =
    2m_{2}\metric{\nabla_{X}Y}{Z} 
    + m_{3}\metric{\mathfrak{R}(u,Y)X}{Z}$\label{it:LeviProp3}
  \item
    $2\bmetric{\bar{\nabla}_{\hlift{X}}\vlift{Y}}{\vlift{Z}} =
    2m_{3}\metric{\nabla_{X}Y}{Z}$\label{it:LeviProp4}
  \item
    $2\bmetric{\bar{\nabla}_{\vlift{X}}\hlift{Y}}{\hlift{Z}} =
    m_{3}\metric{\mathfrak{R}(u,X)Y}{Z}$\label{it:LeviProp5}
  \item
    $2\bmetric{\bar{\nabla}_{\vlift{X}}\hlift{Y}}{\vlift{Z}} = 0$
    \label{it:LeviProp6}
  \item
    $2\bmetric{\bar{\nabla}_{\vlift{X}}\vlift{Y}}{\hlift{Z}} = 0$
    \label{it:LeviProp7}
  \item
    $2\bmetric{\bar{\nabla}_{\vlift{X}}\vlift{Y}}{\vlift{Z}} = 0$
    \label{it:LeviProp8}
  \end{enumerate}
for all vector fields $X,Y,Z \in \mathfrak{X}\left(M\right)$.
\end{proposition}
\begin{proof}
  The Kozul formula on the tangent bundle is used repeatedly to find the properties of
  the Levi-Civita connection.
  \begin{itemize}
    \item[\ref{it:LeviProp1}]
      The statement follows from the Kozul formula in the first equation.
      Then substituting properties from \cref{eq:LieBracketRelations} and
      \cref{eq:NonNaturalMetric}, we obtained the second equation.
      The third equation follows from the fact that six of the terms
      produce the Kozul formula on $M$. Lastly, we obtained the fourth
      equation by combining the Riemannian curvature tensor dependent
      terms such that $Z$ is isolated using the properties of the curvature tensor.
      \begin{multline*}
        2\bmetric{\bar{\nabla}_{\hlift{X}}\hlift{Y}}{\hlift{Z}} =
        \hlift{X}\bmetric{\hlift{Y}}{\hlift{Z}} +
        \hlift{Y}\bmetric{\hlift{Z}}{\hlift{X}}
        - \hlift{Z}\bmetric{\hlift{X}}{\hlift{Y}}\\
        - \bmetric{\hlift{X}}{\bliebracket{\hlift{Y}}{\hlift{Z}}}
        + \bmetric{\hlift{Y}}{\bliebracket{\hlift{Z}}{\hlift{X}}} 
        + \bmetric{\hlift{Z}}{\bliebracket{\hlift{X}}{\hlift{Y}}}\\
        = m_{1}X\metric{Y}{Z} + m_{1}Y\metric{Z}{X}
        - m_{1}Z\metric{X}{Y} -m_{1}\metric{X}{\liebracket{Y}{Z}}\\
        + m_{2}\metric{X}{\mathfrak{R}(Y,Z)u} + m_{1}\metric{Y}{\liebracket{Z}{X}}
        - m_{2}\metric{Y}{\mathfrak{R}(Z,X)u}\\ + m_{1}\metric{Z}{\liebracket{X}{Y}}
        - m_{2}\metric{Z}{\mathfrak{R}(X,Y)u)}\\      
        = 2m_{1}\metric{\nabla_{X}Y}{Z} + m_{2}\metric{X}{\mathfrak{R}(Y,Z)u}
        - m_{2}\metric{Y}{\mathfrak{R}(Z,X)u}\\ -m_{2}\metric{Z}{\mathfrak{R}(X,Y)u}\\
        = 2m_{1}\metric{\nabla_{X}Y}{Z} + 2m_{2}\metric{\mathfrak{R}(u,X)Y}{Z}\\
      \end{multline*}
    \item[\ref{it:LeviProp2}]
      The statement is obtained in a similar fashion to \ref{it:LeviProp1}.
      The first equation is the Kozul formula. The second equation is
      obtained by substituting properties from \cref{eq:LieBracketRelations}
      and \cref{eq:NonNaturalMetric} followed by the expansion of the derivative
      of the metric terms using the metric compatibility. Note
      that by \cref{eq:NonNaturalMetric}, we can choose $g(X,Y)$ to be
      purely horizontal or vertical. Thus, $\vlift{Z}\metric{X}{Y} = 0$.
      Finally, the last equation is obtained by expanding the Lie Bracket
      and combining terms.
      \begin{multline*}
        2\bmetric{\bar{\nabla}_{\hlift{X}}\hlift{Y}}{\vlift{Z}} =
        \hlift{X}\bmetric{\hlift{Y}}{\vlift{Z}} +
        \hlift{Y}\bmetric{\vlift{Z}}{\hlift{X}}
        - \vlift{Z}\bmetric{\hlift{X}}{\hlift{Y}}\\
        - \bmetric{\hlift{X}}{\bliebracket{\hlift{Y}}{\vlift{Z}}}
        + \bmetric{\hlift{Y}}{\bliebracket{\vlift{Z}}{\hlift{X}}}
        + \bmetric{\vlift{Z}}{\bliebracket{\hlift{X}}{\hlift{Y}}}\\
        = m_{2}\metric{Z}{\nabla_{Y}X} + m_{2}\metric{X}{\nabla_{Y}Z}
        + m_{2}\metric{Y}{\nabla_{X}Z}
        + m_{2}\metric{Z}{\nabla_{X}Y}\\ - m_{2}\metric{X}{\nabla_{Y}Z}
        + m_{2}\metric{Y}{-\nabla_{X}Z} + m_{2}\metric{Z}{\liebracket{X}{Y}}
        - m_{3}\metric{Z}{\mathfrak{R}(X,Y)u}\\
        = 2m_{2}\metric{\nabla_{X}Y}{Z}-m_{3}\metric{\mathfrak{R}(X,Y)u}{Z}\\
      \end{multline*}
      \ref{it:LeviProp3}-\ref{it:LeviProp7} are analogous to \ref{it:LeviProp2}.
    \item[\ref{it:LeviProp8}]
      The statement follows from the result that the Lie bracket of 
      two vertical vector fields vanish and that $\metric{\cdot}{\cdot}$
      can be chosen to be purely horizontal or vertical.
      \begin{multline*}
        2\bmetric{\bar{\nabla}_{\vlift{X}}\vlift{Y}}{\vlift{Z}} =
        \vlift{X}\bmetric{\vlift{Y}}{\vlift{Z}} +
        \vlift{Y}\bmetric{\vlift{Z}}{\vlift{X}}
        - \vlift{Z}\bmetric{\vlift{X}}{\vlift{Y}}\\
        - \bmetric{\vlift{X}}{\bliebracket{\vlift{Y}}{\vlift{Z}}}
        + \bmetric{\vlift{Y}}{\bliebracket{\vlift{Z}}{\vlift{X}}}
        + \bmetric{\vlift{Z}}{\bliebracket{\vlift{X}}{\vlift{Y}}}\\
        = m_{3}\vlift{X}\metric{Y}{Z} + m_{3}\vlift{Y}\metric{Z}{X}
        - m_{3}\vlift{Z}\metric{X}{Y}\\
        - \bmetric{\vlift{X}}{0} + \bmetric{\vlift{Y}}{0}
        + \bmetric{\vlift{Z}}{0}\\
        = 0\\
      \end{multline*}      
  \end{itemize}
\end{proof}

Next, we extract the explicit form of the horizontal and vertical components of
the Levi-Civita connection on the tangent bundle from \cref{lem:KozulFormula}.
To do so, we first present a useful lemma.
\begin{lemma}
\label{lem:idMetricTM}
Let $\bar{f}$ be a function $\bar{f}: T_{\left(p,u\right)}TM \rightarrow T_{\left(p,u\right)}TM$ such that
\begin{equation}
  \begin{split}
    &\hcomp{\bar{f}\circ\bar{X}} = \frac{m_{3}\hcomp{\bar{X}}-m_{2}\vcomp{\bar{X}}}
    {m_{1}m_{3}-m_{2}^{2}}\\
    &\vcomp{\bar{f}\circ\bar{X}} = \frac{-m_{2}\hcomp{\bar{X}}+m_{1}\vcomp{\bar{X}}}
    {m_{1}m_{3}-m_{2}^{2}}\\
  \end{split}
\end{equation}
for all vector fields $\bar{X} \in \mathfrak{X}\left(TM\right)$, 
$\left(p,u\right) \in TM$, and $m_{1},m_{2},m_{3} \in \mathbb{R}$ such that
$m_{1}, m_{3} > 0$ and $m_{1}m_{3}-m_{2}^{2}>0$. Then
\begin{equation}
  \begin{split}
    \bmetric[\left(p,u\right)]{\bar{f}\circ\bar{X}}{\bar{Y}} ={ }&
    \bmetric[\left(p,u\right)]{\bar{X}}{\bar{f}\circ\bar{Y}}\\
    ={ }&\metric[p]{\hcomp{\bar{X}}}{\hcomp{\bar{Y}}}
    + \metric[p]{\vcomp{\bar{X}}}{\vcomp{\bar{Y}}}.
  \end{split}
\end{equation}
\end{lemma}
\begin{proof}
  The claim follows directly from the definitions of $\bar{f}$ and $\bar{g}$.
\end{proof}
\cref{lem:idMetricTM} is important in that if there is an expression for $\bar{g}$
with one known tangent vector, the horizontal and vertical components of the
unknown tangent vector can be extracted through the metric instead of
the $d\pi$ and $K$ mappings.
\begin{remark}
The results of \cref{lem:idMetricTM} can be better understood in
local coordinates using matrix operations. To illustrate the
point, we assume $g$ to be the natural Euclidean inner product on $M$, then
\begin{equation}
  \label{eq:NonNaturalMetricMat}
  \begin{split}
    \bmetric[\left(p,u\right)]{\bar{X}}{\bar{Y}}={ }&
    \begin{bmatrix}
      \hcomp{\bar{X}}\\
      \vcomp{\bar{X}}
    \end{bmatrix}^{T}
    \begin{bmatrix}
      m_{1}\mathcal{I}_{n} & m_{2}\mathcal{I}_{n}\\
      m_{2}\mathcal{I}_{n} & m_{3}\mathcal{I}_{n}
    \end{bmatrix}
    \begin{bmatrix}
      \hcomp{\bar{Y}}\\
      \vcomp{\bar{Y}}
    \end{bmatrix}
    \Biggr|_{p}\\
    ={ }&
    \begin{bmatrix}
      \hcomp{\bar{X}}\\
      \vcomp{\bar{X}}
    \end{bmatrix}^{T}
    \mathcal{M}
    \begin{bmatrix}
      \hcomp{\bar{Y}}\\
      \vcomp{\bar{Y}}
    \end{bmatrix}
    \Biggr|_{p}
  \end{split}
\end{equation}
where $\mathcal{I}_{n}$ is the $n \times n$ identity matrix, $\bar{X}, \bar{Y} \in T_{\left(p,u\right)}TM$,
$\left(p,u\right) \in TM$, and $m_{1},m_{2},m_{3}$ $\in \mathbb{R}$ such that
$m_{1},m_{3} > 0$ and $m_{1}m_{3}-m_{2}^{2} > 0$. Since $\mathcal{M}$ is positive definite,
its inverse $\mathcal{M}^{-1}$ exists.
Thus, the function $\bar{f}$ can be interpreted (in matrix notation) as 
\begin{equation}
  \bar{f}\circ\bar{X} = \mathcal{M}^{-1}
  \begin{bmatrix}
    \hcomp{\bar{X}}\\
    \vcomp{\bar{X}}
  \end{bmatrix}.
\end{equation}
When $\bar{f}$ acts on a tangent vector in \cref{eq:NonNaturalMetricMat}, we 
recover the identity matrix and the simple pairing of the horizontal and 
vertical components.
\end{remark}

The following theorem combines \cref{lem:KozulFormula} and \cref{lem:idMetricTM}
to extract the explicit form of the horizontal and vertical components of
the Levi-Civita connection $\bar\nabla_{\bar{X}}\bar{Y}$ on $TM$ for any
vector fields $\bar{X}, \bar{Y} \in \mathfrak{X}(TM)$.

\begin{theorem}
  \label{thm:LeviCivitaConn}
  Let (M, g) be a Riemannian manifold and $\bar{\nabla}$ be the
  Levi-Civita connection on the tangent bundle (TM, $\bar{g}$)
  equipped with the metric \cref{eq:NonNaturalMetric}.  Then
\end{theorem}
\begin{enumerate}[label=(\roman*)]
\item
  $\hcomp{\bar{\nabla}_{\hlift{X}}\hlift{Y}} = \nabla_{X}Y +
  \frac{1}{m_{1}m_{3}-m_{2}^{2}}\Big(m_{2}m_{3}\mathfrak{R}(u,X)Y +
  \frac{m_{2}m_{3}}{2}\mathfrak{R}(X,Y)u \Big)$
\item
  $\vcomp{\bar{\nabla}_{\hlift{X}}\hlift{Y}} =
  \frac{1}{m_{1}m_{3}-m_{2}^{2}}\Big( -m_{2}^{2}\mathfrak{R}(u,X)Y
  -\frac{m_{1}m_{3}}{2}\mathfrak{R}(X,Y)u \Big)$
\item
  $\hcomp{\bar{\nabla}_{\hlift{X}}\vlift{Y}} =
  \frac{1}{m_{1}m_{3}-m_{2}^{2}}\Big( \frac{m_{3}^{2}}{2}\mathfrak{R}(u,Y)X
  \Big)$
\item
  $\vcomp{\bar{\nabla}_{\hlift{X}}\vlift{Y}} = \nabla_{X}Y -
  \frac{1}{m_{1}m_{3}-m_{2}^{2}}\Big(\frac{m_{2}m_{3}}{2}\mathfrak{R}(u,Y)X
  \Big)$
\item
  $\hcomp{\bar{\nabla}_{\vlift{X}}\hlift{Y}} =
  \frac{1}{m_{1}m_{3}-m_{2}^{2}}\Big( \frac{m_{3}^{2}}{2}\mathfrak{R}(u,X)Y
  \Big)$
\item
  $\vcomp{\bar{\nabla}_{\vlift{X}}\hlift{Y}} =
  -\frac{1}{m_{1}m_{3}-m_{2}^{2}}\Big( \frac{m_{2}m_{3}}{2}\mathfrak{R}(u,X)Y
  \Big)$
\item $\hcomp{\bar{\nabla}_{\vlift{X}}\vlift{Y}} = 0$
\item $\vcomp{\bar{\nabla}_{\vlift{X}}\vlift{Y}} = 0$
\end{enumerate}
for all vector fields $X,Y \in \mathfrak{X}\left(M\right)$ and $\left(p,u\right) \in TM$.\\

\begin{proof}
  \cref{lem:KozulFormula} provides an expression for 
  $\bmetric[\left(p,u\right)]{\bar{\nabla}_{\bar{X}}\bar{Y}}{\cdot}$ for any vector fields
  $\bar{X}, \bar{Y} \in \mathfrak{X}\left(TM\right)$ at a point $\left(p,u\right) \in TM$
  where the second argument of $\bar{g}$ can be chosen arbitrarily. Thus,
  we chose a purely horizontal and vertical field to extract the components of the
  connection. For any arbitrary vector field $Z \in \mathfrak{X}(M)$ and
  $\bar{f}$ defined in \cref{lem:idMetricTM}
  \begin{align*}
    \bmetric[\left(p,u\right)]{\bar{\nabla}_{\bar{X}}\bar{Y}}{\bar{f}\circ\hlift{Z}}
    &= \metric[p]{\hcomp{\bar{\nabla}_{\bar{X}}\bar{Y}}}{Z}\\
    \bmetric[\left(p,u\right)]{\bar{\nabla}_{\bar{X}}\bar{Y}}{\bar{f}\circ\vlift{Z}}
    &=\metric[p]{\vcomp{\bar{\nabla}_{\bar{X}}\bar{Y}}}{Z}.\\
  \end{align*}
\end{proof}

The results of this section allows us to compute the Levi-Civita connection
on the tangent bundle for any vector fields $\bar{X}, \bar{Y} \in \mathfrak{X}(TM)$
that are lift decomposable. However, lift-decomposable vector fields
do not span the space of all possible smooth vector fields. In general,
vector fields on the tangent bundle may change along both horizontal
and vertical curves. In the next section, we show how to extend
the results for lift-decomposable vector fields to any general
vector field.

\section{Levi-Civita Connection for General Vector Fields}
\label{sec:main2}
In this section, we extend the Levi-Civita connection in \cref{sec:main1} to general
vector fields on the tangent bundle that may change along
both horizontal and vertical curves. As discussed in \cref{sec:background}
and \cref{sec:main1}, the Levi-Civita connection in \cref{thm:LeviCivitaConn}
is only valid for lift-decomposable vector fields. In general, vector fields
$\bar{Y} \in \mathfrak{X}\left(TM\right)$ at a point $(p,u) \in TM$ depend on both
horizontal and vertical motions and may be expressed as
\begin{equation}
  \label{eq:GenVF}
  \bar{Y}_{(p,u)} = \hlift{A}_{(p,u)} + \vlift{B}_{(p,u)} + \bar{C}_{(p,u)} + \bar{D}_{(p,u)}
\end{equation}
where $A,B \in T_{p}M$, $\bar{C} \in \mathcal{H}_{(p,u)}$, $\bar{D} \in \mathcal{V}_{(p,u)}$, 
and $\bar{C} = \bar{D} = 0$ at $\left(p,u\right)$ and along
horizontal curves passing through $(p,u)$. To be more precise,
$\bar{C}$ and $\bar{D}$ are the point-wise horizontal and vertical
projections of the field $\bar{Y}_{(p',u')}-\hlift{A}_{(p,u)}-\vlift{B}_{(p,u)}$
for any point $(p',u') \in TM$ in the neighborhood around $(p,u)$. It is important
to note that $A$ and $B$ change along horizontal curves, and $\bar{C}$ and $\bar{D}$
change along vertical curves.

The standard results and our results in \cref{sec:main1}
already considered how vector fields change along horizontal
curves to derive the connection in \cref{thm:LeviCivitaConn}. In that
formulation, we ignored the motion along the vertical curves because the vector fields
are lift decomposable and thus constant along those curves. Now, we must also consider
changes along vertical curves to obtain the Levi-Civita connection for general vector fields
on the tangent bundle.
\begin{corollary}
  \label{cor:TotalDerTM}
  The Levi-Civita connection $\bar{\nabla}_{\bar{X}}\bar{Y}$ on the tangent bundle $TM$ for any
  general vector fields $\bar{X}, \bar{Y} \in \mathfrak{X}\left(TM\right)$ at
  a point $(p,u) \in TM$ is given by
  \begin{equation}
    \bar{\nabla}_{\bar{X}} \bar{Y} =
    \bar{\nabla}_{(\hlift{F}+\vlift{G})}\left(\hlift{A} + \vlift{B}\right)
    + \tilde{\nabla}_{\vlift{G}}\left(\bar{C}+\bar{D}\right)\\
  \end{equation}
  where $\bar{Y}$ is decomposed into the components defined in \cref{eq:GenVF}
  and $\bar{X} = \hlift{F} + \vlift{G}$ for $F,G \in T_{p}M$.
  The first term is the connection from \cref{thm:LeviCivitaConn} which captures
  changes along horizontal curves. The second term captures changes along
  vertical curves and does not depend on $\hlift{F}$ since $\bar{C}, \bar{D}$ are
  zero along any horizontal curve. The connection $\tilde{\nabla}$ is the usual
  connection on the flat tangent space corresponding to the choice of local coordinates.
\end{corollary}
\begin{proof}
  The proof follows from the vector field decomposition in \cref{eq:GenVF} and
  the properties of the Levi-Civita connection. Note that since $\bar{C}, \bar{D} = 0$
  along horizontal curves, the connection
  $\tilde{\nabla}_{\hlift{F}}\left(\bar{C}+\bar{D}\right) = 0$.
\end{proof}

\section{Examples}
In this section, we present two applications of our results.
In the first, we show that the Sasaki metric and the corresponding Levi-Civita
connection on $TM$ is a special case. In the second, we apply the results on
SO(3) and derive the Levi-Civita connection on $T$SO(3).
\label{sec:experiments}
\subsection{Sasaki Metric}
In this example, we show that the Saski Metric \cite{Sasaki1958} and the induced Levi-Civita
connection on $TM$ is a special case of our results. If we choose
\begin{equation*}
  m_{1} = 1, \quad m_{2} = 0, \quad m_{3} = 1
\end{equation*}
then the metric \cref{eq:NonNaturalMetric} becomes
\begin{equation*}
  \begin{split}
    &\bmetric{\hlift{X}}{\hlift{Y}} = \metric{X}{Y}\\
    &\bmetric{\hlift{Y}}{\vlift{Y}} = 0\\
    &\metric{\vlift{X}}{\vlift{Y}} = \metric{X}{Y}.
  \end{split}
\end{equation*}
The induced connection $\bar{\nabla}$ on $TM$, given by \cref{thm:LeviCivitaConn},
can be shown to be equivalent to the results obtained by Kowalski in \cite{Kowalski1971}.
\subsection{SO(3) Example}
In this example, we consider the Special Orthogonal Group $SO(3)$ equipped with a metric $g$
and its tangent bundle $TSO(3)$ equipped with the metric in \cref{eq:NonNaturalMetric}.
The Levi-Civita connection on $SO(3)$ is given by Edelman 
in \cite{Edelman1998}
\begin{equation}
  \nabla_{X}Y = \dot{Y} + \frac{1}{2}R\left(X^{T}Y + Y^{T}X\right)
\end{equation}
for all vector fields $X,Y \in \mathfrak{X}\left(SO(3)\right)$ at a point 
$R \in SO(3)$ and $\dot{Y}$ is the usual time derivative. 
Given left-invariant vector fields
$\bar{X}, \bar{Y} \in \mathfrak{X}\left(TSO(3)\right)$ along a curve
$\bar{\gamma}$ such that
\begin{equation}
  \bar{X} = \left(R\hat{\zeta}, R\hat{\eta} \right),
  \quad
  \bar{Y} = \left( R\hat{\alpha}, R\hat{\beta} \right),
  \quad
  \bar{\gamma} = \left(R, R\hat{\omega}\right)
\end{equation}
where constants $\zeta, \eta, \alpha, \beta, \omega \in \mathbb{R}^{3}$,
$\hat{(\cdot)}:\mathbb{R}^{3} \rightarrow so(3)$ is the hat operator which
map real numbers to the Lie algebra, and $TT_{R}SO(3) \rightarrow
T_{R}SO(3)$. Then the induced Levi-Civita
connection $\bar{\nabla}$ on $TSO(3)$, in local coordinates, is
given by
\begin{enumerate}[label=(\roman*)]
\item
  \begin{equation*}
    \begin{split}
      \hcomp{\bar{\nabla}_{\bar{X}}\bar{Y}} ={ }&R\left(\hat{\zeta}\hat{\alpha} +
        \frac{1}{2}\left(\hat{\zeta}^{T}\hat{\alpha}
          +\hat{\alpha}^{T}\hat{\zeta}\right)\right)\\
        &\quad -R\frac{m_{2}m_{3}}{8(m_{1}m_{3}-m_{2}^{2})}\left(
          2\liebracket{\liebracket{\hat{\omega}}{\hat{\zeta}}}{\hat{\alpha}}
          +\liebracket{\liebracket{\hat{\zeta}}{\hat{\alpha}}}{\hat{\omega}}
        \right)\\
        &\quad -R\frac{m_{3}^{2}}{8(m_{1}m_{3}-m_{2}^{2})}
        \liebracket{\liebracket{\hat{\omega}}{\hat{\beta}}}{\hat{\zeta}}\\
        &\quad -R\frac{m_{3}^{2}}{8(m_{1}m_{3}-m_{2}^{2})}
        \liebracket{\liebracket{\hat{\omega}}{\hat{\eta}}}{\hat{\alpha}}
    \end{split}
  \end{equation*}
\item
  \begin{equation*}
    \begin{split}
      \vcomp{\bar{\nabla}_{\bar{X}}\bar{Y}} ={ }&R\left(\hat{\zeta}\hat{\beta} +
        \frac{1}{2}\left(\hat{\zeta}^{T}\hat{\beta}
          +\hat{\beta}^{T}\hat{\zeta}\right)\right)\\
      &\quad +R\frac{m_{2}m_{3}}{8(m_{1}m_{3}-m_{2}^{2})}
        \liebracket{\liebracket{\hat{\omega}}{\hat{\beta}}}{\hat{\zeta}}\\
      &\quad  +R\frac{1}{8(m_{1}m_{3}-m_{2}^{2})}\left(
        2m_{2}^{2}\liebracket{\liebracket{\hat{\omega}}{\hat{\zeta}}}{\hat{\alpha}}
        +m_{1}m_{3}\liebracket{\liebracket{\hat{\zeta}}{\hat{\alpha}}}{\hat{\omega}}
      \right)\\
      &\quad +R\frac{m_{2}m_{3}}{8(m_{1}m_{3}-m_{2}^{2})}
    \liebracket{\liebracket{\hat{\omega}}{\hat{\eta}}}{\hat{\alpha}}.
    \end{split}
  \end{equation*}
\end{enumerate}

In the general case where $\omega = \omega(t)$, $\alpha = \alpha(\omega)$,
$\beta = \beta(\omega)$, an additional term is required to account for
changes in the vertical subspace along the curve $\bar{\gamma}$
(see \cref{cor:TotalDerTM}). The connection on $TSO(3)$ for
this vector field is given by
\begin{enumerate}[label=(\roman*)]
\item
  \begin{equation*}
    \hcomp{\bar{\nabla}_{\bar{X}}\bar{Y}} = ... + R\left(
      \frac{\partial\alpha}{\partial\omega}\dot{\omega}\right)\hat{}
    \text{  }\hat{\eta}
  \end{equation*}
\item
  \begin{equation*}
    \vcomp{\bar{\nabla}_{\bar{X}}\bar{Y}} = ... + R\left(
      \frac{\partial\beta}{\partial\omega}\dot{\omega}\right)\hat{}
    \text{  }\hat{\eta}.
  \end{equation*}
\end{enumerate}
where $\tilde{\nabla}$ is the usual directional derivative on $\mathbb{R}^{3}$.
Both results can be validated by the metric compatibility requirement
along their respective curves.

\section{Conclusions}
\label{sec:conclusions}
In this paper, we study the relationship between Riemannian manifolds
and their tangent bundle. Namely, we see that a manifold equipped with a 
metric and Levi-Civita connection induces a metric and Levi-Civita
connection on its tangent bundle by the natural decomposition
of the tangent bundle into the horizontal
and vertical subspaces. We then defined a non-natural metric
on the tangent bundle and derived the corresponding Levi-Civita
connection. In addition, we showed explicitly how to
extend the results to vector fields that are not constant along
the fibers. As a validation of our results, we see that
under special conditions the non-natural metric reduces
to the Sasaki metric and the corresponding Levi-Civita
connection agrees with the results of Kowalski.




\bibliographystyle{siamplain} \bibliography{references}

\begin{thebibliography}{1}

\bibitem{BulloFrancesco2005Gcom}
{\sc F.~Bullo and A.~D. Lewis}, {\em Geometric control of mechanical systems :
  modeling, analysis, and design for simple mechanical control systems}, Texts
  in applied mathematics ; 49, Springer, New York, 2005.

\bibitem{Dombrowski1962}
{\sc P.~Dombrowski}, {\em {On the Geometry of the Tangent Bundle}}, Journal fur
  die reine und angewandte Mathematik, 210 (1962), pp.~73--88.

\bibitem{Edelman1998}
{\sc A.~Edelman, T.~A. Arias, and S.~T. Smith}, {\em {The Geometry of
  Algorithms with Orthogonality Constraints}}, SIAM Journal on Matrix Analysis
  and Applications, 20 (1998), pp.~303--353,
  \url{https://doi.org/10.1137/S0895479895290954},
  \url{http://epubs.siam.org/doi/abs/10.1137/S0895479895290954},
  \url{https://arxiv.org/abs/9806030}.

\bibitem{Gudmundsson2002}
{\sc S.~Gudmundsson and E.~Kappos}, {\em {On the Geometry of Tangent Bundles}},
  Expositiones Mathematicae, 20 (2002), pp.~1--41,
  \url{https://doi.org/10.1016/S0723-0869(02)80027-5}.

\bibitem{Kowalski1971}
{\sc O.~Kowalski}, {\em {Curvature of the Induced Riemannian Metric on the
  Tangent Bundle of a Riemannian Manifold}}, Journal fur die reine und
  angewandte Mathematik, 250 (1971), pp.~124--129.

\bibitem{Sasaki1958}
{\sc S.~Sasaki}, {\em {On the differential geometry of tangent bundles of
  Riemannian manifolds}}, Tohoku Mathematical Journal, 10 (1958), pp.~338--354,
  \url{https://doi.org/10.2748/tmj/1178244668}.

\end{thebibliography}
\end{document}